\documentclass{article}%
\usepackage{amsmath}
\usepackage{amsfonts}
\usepackage{amssymb}
\usepackage{graphicx}%
\providecommand{\U}[1]{\protect\rule{.1in}{.1in}}
\newtheorem{theorem}{Theorem}

\newtheorem{corollary}[theorem]{Corollary}

\newtheorem{definition}[theorem]{Definition}

\newtheorem{lemma}[theorem]{Lemma}

\newtheorem{remark}[theorem]{Remark}

\newenvironment{proof}[1][Proof]{\noindent\textbf{#1.} }{\ \rule{0.5em}{0.5em}}
\begin{document}

\title{The $\mu$-neutral fractional multi-delayed differential equations}
\author{Mustafa Aydin and Nazim I. Mahmudov\\Department of Mathematics, Eastern Mediterranean University \\Famagusta, 99628, T. R. Northern Cyprus via Mersin 10, Turkey}
\date{}
\maketitle

\begin{abstract}
The $\mu$-neutral linear fractional multi-delayed differential nonhomogeneous system with noncommutative coefficient matrices is introduced. The novel $\mu$-neutral multi-delayed perturbation of Mittag-Leffler type matrix function is proposed. Based on this, an explicit solution to the system is investigated step by step. The existence uniqueness of solutions to $\mu$-neutral nonlinear fractional multi-delayed differential system is obtained with regard to the supremum norm. The notion of stability analysis in the sense of solutions to the described system is discussed on the grounds of the fixed point approach.

\end{abstract}



\section{ Introduction and Preliminaries }

Fractional Calculus which is seen as a generalisation of the ordinary calculus  has drawn many of researchers' attention in recent years. And also it has been started to be exponentially utilized in so many different kinds of areas like finance, engineering,  neurons, electric conductance, diffusion, thermodynamics, computed tomography, mechanism, mathematical physics.

Fractional(ordinary) delay differential equations\cite{i2}\cite{i3}\cite{Aydinsssss}\cite{i5}\cite{1}  are of a considerable importance since they permit to describe such systems that the rate of change depends both on the present and delayed states and on the past state unlike other systems. They are very often exploited in propagation of energy or information and transport in interdependent systems (see e.g. \cite{X},\cite{Y},\cite{Z} and reference therein). In reference \cite{i1}, Khusainov and  Shuklin  were able to find  a solution of the below linear fractional delayed differential system in terms of producing the delayed exponential matrix.
\begin{equation}
\left\{
\begin{array}
[c]{l}%
 w^{'} \left( t\right) =Bw\left(  t\right)  +Fw\left(
t-r\right)     ,\ \ t>0 \ \ r>0 \ \ \text{(delay)}  ,\\
\ \ \ \ w\left(  t\right)  =\phi\left(  t\right)  ,\ \ -r\leq t\leq 0.\\
\end{array}
\right.  \label{birdenklem}
\end{equation}
The drawback of  system (\ref{birdenklem}) is to assume that the coefficient matrices $B$ and $F$ must be commutative.  In reference\cite{i2}, Li and Wang consider the fractional version of system (\ref{birdenklem}) in the case of $F=\Theta$. The drawback is also in progress because of validation of commutativity of coefficient matrices when $F=\Theta$. In reference\cite{i3},  Mahmudov investigates a general version of fractional delay differential system whose coefficient matrices do not need to be either commutative or zero. Having proposed a newly delay perturbation of Mittag-Leffler type matrix function, Mahmudov gives an excellent representation of solutions of the mentioned system. In the sequel, Mahmudov extends the fractional delay differential equations in the work\cite{i3} to multi-delayed version in the work\cite{i5} and solve it.

In the neutral version of the fractional(ordinary) delay differential equations \cite{i9},
\cite{6}, \cite{m10} the fractional(ordinary) derivative of the unknown function appears mostly with delays and rarely without delays. These kinds of systems are used from population growth to the motion of radiation electrons, spread of epidemic. In reference\cite{6}, Pospisil and Skripkova investigate the following linear neutral fractional delay differential equations
\begin{equation}\label{int5}
\left\{
\begin{array}
[c]{l}%
 w^{'} \left( t\right)- Aw^{'} \left( t-r\right) =Bw\left(  t\right)  +Fw\left(
t-r\right)+f(t)     ,\ \ t>0 \ \ r>0 \ \ \  ,\\
\ \ \ \ w\left(  t\right)  =\phi\left(  t\right)  ,\ \ -r\leq t\leq 0,\\
\end{array}
\right.
\end{equation}
where $f$ is continuous from $[0,\infty)$ to $\mathbb{R}^n$, $\phi$ is continuously differentiable form $[-r,0]$ to $\mathbb{R}^n$  and $r$ is a retardation. The coefficient matrices $A, B, F$ are permutable, that is $AB=BA$, $AF=FA$,
$BF=FB$. Zhang et al.\cite{i9} look into the representation of the solution to the neutral fractional linear differential system having a constant delay
\begin{equation}\label{int3}
\begin{array}
[c]{l}%
{^{\mathfrak{C}}\beth^{\alpha}_{0^+}}\left( w \left( t\right)- Aw\left( t-r\right) \right)=Bw\left( t\right)  +Fw\left(
t-r\right)+f(t)     ,\ \ t>0,\\
\ \ \ \ w\left(  t\right)  =\phi\left(  t\right)  ,\ \ -r\leq t\leq 0,\\
\end{array}
\end{equation}
where ${^{\mathfrak{C}}\beth^{\alpha}_{0^+}}$ is Caputo fractional derivative of order $\alpha$, $0<\alpha<1$, $r>0$, $f$ is continuous from $[0,\infty)$ to $\mathbb{R}^n$ $A,B,F \in \mathbb{R}^{n \times n}$,  and $\phi$ is continuously differentiable form $[-r,0]$ to $\mathbb{R}^n$. In an attempt to solve system (\ref{int3}),  Zhang et al.\cite{i9} exploited Laplace integral transform. This produced some drawback and mistakes because the representation of power series of the fundamental solution is unknowable.

In reference\cite{Almeida}, Almeida consider a Caputo type fractional derivative with respect to another function. Some features between the fractional derivative and integral, Fermat's Theorem, Taylor’s Theorem, and a lot more are  studied. In reference\cite{Almeidaveark}, Almeida et al. questionnaire existence and uniqueness results for the initial value problem of nonlinear fractional differential equations involving a Caputo-type fractional derivative with respect to another function with the help of the some standard fixed point theorems and develop the Picard iteration method for solving numerically the problem and obtain results on the long-term behavior of solutions.

Motivated by the above cited studies, we handle the below $\mu$-Caputo type (more general) fractional neutral differential multi-delayed equations with noncommutative matrices
\begin{equation}
\begin{array}
[c]{l}%
{_{0^+}^{{C}}\beth^{\alpha}_{\mu}}   \left[ w \left( t\right)- {\displaystyle\sum_{i=1}^{d}} A_{i} w \left( t-r_{i}\right) \right]  =Bw\left( t\right)  +{\displaystyle\sum_{i=1}^{d}}F_{i}w\left(t-r_{i}\right)  +\daleth\left(  t\right), t\in\left(  0,T\right],\\
\ \ \ \ w\left(  t\right)  =\varphi\left(  t\right)  ,\ \ -r\leq t\leq0,\\
\end{array}  \label{sorkok}%
\end{equation}
where the Caputo fractional derivative ${_{0^+}^{{C}}\beth^{\alpha}_{\mu}}$ is of order $\alpha \in \left(0,1  \right) $. For each of $i=1,2,3, \dots, d$,   $A_{i}$,   $B$, $C_{i}$ are square coefficient constant matrices  which do not need to be permutable and $r_{i}>0$ is a retardation and   $r:=\max\{r_i: i=1,2,3, \dots, d\}$. An arbitrary vector function $\phi\left(  x\right)$ is continuously differentiable and $\daleth \in C\left( \left[0, T \right],  \mathbb{R}^n  \right)$ with $T=ld$ for a fixed $l\in \mathbb{N}$. After finding the explicit solutions of (\ref{sorkok}), we obtain the explicit solutions to the below equations (\ref{sorkok2})
\begin{equation}
\begin{array}
[c]{l}%
{_{0^+}^{{C}}\beth^{\alpha}_{\mu}}   \left[ w \left( t\right)- {\displaystyle\sum_{i=1}^{d}} A_{i} w \left( t-r_{i}\right) \right]  =Bw\left( t\right)  +{\displaystyle\sum_{i=1}^{d}}F_{i}w\left(t-r_{i}\right)  +\daleth\left(  t,w(t)\right),\\
\ \ \ \ w\left(  t\right)  =\varphi\left(  t\right)  ,\ \ -r\leq t\leq0,\\
\end{array} \label{sorkok2}%
\end{equation}
where $\daleth \in C\left( \left[ 0, T \right] \times \mathbb{R}^n ,  \mathbb{R}^n  \right)$    and the others are the same as (\ref{sorkok}).

\begin{remark} By choosing $A_i=\Theta$, $i=1,2,3, \dots, d$ and $\mu(t)=t$, $\mu$-Caputo fractional neutral differential multi-delayed equations with nonpermutable  matrices overlaps with fractional linear multi-delay differential equations in the study \cite{i5}.
\end{remark}

In the current work,
\begin{itemize}
  \item we introduce the $\mu$-neutral Caputo type fractional linear(or semi-linear) multi-delayed differential equations with non-permutable constant coefficient matrices,
  \item we propose newly the $\mu$-neutral multi-delayed perturbation of two parameter Mittag-Leffler type matrix function,
  \item we present a representation of a solution for the $\mu$-neutral Caputo type fractional linear(or semi-linear) multi-delayed differential equations with non-permutable constant coefficient matrices by sharing  the $\mu$-neutral multi-delayed perturbation of two parameter Mittag-Leffler type matrix function,
  \item we examine the existence and uniqueness of solutions of the $\mu$-neutral fractional multi-delayed differential equations' system,
  \item we show the stability of the $\mu$-neutral fractional order multi-delayed differential system in the sense of Ulam-Hyers, and illustrate the theoretical findings.
\end{itemize}

Now we remind a couple of well-recognized basic notions in the literature.

For $n \in  \{1,2,3,... \}$, the space $C^n(\left[  0,T\right]  ,\mathbb{R}^{n})$ is all of
continuously $n$th order differentiable vector-valued functions from $\left[  0,T\right]  $ to
$\mathbb{R}^{n}$  with  $\left\Vert x\right\Vert
_{C}:=\sup_{t\in\left[  0,T\right]  }\left\Vert x(t)\right\Vert \ $ for a norm
$\left\Vert .\right\Vert $ on $\mathbb{R}^{n}$. For $n=0$,  $C(\left[  0,T\right]  ,\mathbb{R}^{n})=C^0(\left[  0,T\right]  ,\mathbb{R}^{n})$ is all of continuous vector-valued functions. Let $AC[0, T]$ be the space of functions  which are absolutely continuous on $[0,T]$.  We denote by $AC^n[0, T]$ the space of complex-valued functions $f(x)$ which have continuous derivatives up to order $n -1$ on $[0, T]$ such  that $f^{(n-1)}(x) \in AC[0,T]$. Let  $\daleth$ and an increasing function $\mu$ on $[0, T]$ be   integrable and continuously differentiable, respectively and let $\mu'(t)\neq 0$ \ $t\in [0, T]$. $\mu$-Riemann-Liouville fractional integrals\cite{sirvastava}\cite{Almeida}  of $\daleth \in AC^n[0, T]$ of order $\alpha\in \mathbb{R^{+}}$ and $n\in \mathbb{N}$ are given by
\begin{eqnarray*}
  \left({}_{0^{+}}^{RL}\gimel_{\mu}^{\alpha}\daleth \right)(t) &:=& \frac{1}{\Gamma(\alpha)}\int_{0}^{t}\left(\mu(t)-\mu(s)\right)^{\alpha-1}\daleth(s)d\mu(s),
\end{eqnarray*}
where $\Gamma\left(  \alpha \right)  = \int_{0}^{\infty} t^{\alpha-1}e^{-t}dt$ with $ Re\left(  \alpha \right)  >0$. $\mu$-Riemann-Liouville fractional derivatives\cite{sirvastava}\cite{Almeida} of $\daleth \in AC^n[0, T]$ of order $\alpha >0$ are given by
\begin{eqnarray*}
 \left( {}_{0^{+}}^{RL}\beth_{\mu}^{\alpha}\daleth\right)(t)    &=& \frac{1}{\Gamma(n-\alpha)}\left(\frac{d}{d\mu (t)}\right)^{n}
   \int_{0}^{t}\left(\mu(t)-\mu(s)\right)^{n-\alpha-1}f(s)d\mu(s),
\end{eqnarray*}
where $n=[\alpha]+1$.  If $\daleth\in AC^{n}\left([0, T],\mathbb{R}\right),\mu\in C^{n}\left([0, T],\mathbb{R}\right)$ with $\mu$ is increasing and $\mu'(t)\neq 0$ for every $t\in [0, T]$, then the $\mu$-Caputo fractional derivatives\cite{sirvastava}\cite{Almeida} of $\daleth$ of order $\alpha$ is defined as
\begin{equation}
  {}_{0^{+}}^{\ C}\beth_{\mu}^{\alpha}\daleth(t)
  :={}_{0^{+}}^{RL}\gimel_{\mu)}^{n-\alpha}\left(\frac{1}{\mu'(t)}\frac{d}{dt}\right)^{n}\daleth(t).
\end{equation}
From \cite{sirvastava} and \cite{Almeida}, We have, for $\mathfrak{R}(\alpha)\geq 0$, $\mathfrak{R}(\beta)> 0$, and $\alpha \in \left(0,1  \right) $,
\begin{equation*}
  \left({}_{0^{+}}^{\ C}\beth_{\mu}^{\alpha}\daleth\right)(t)
  ={}_{0^{+}}^{\ RL}\beth_{\mu}^{\alpha}\left[\daleth(t)-\daleth(0)\right],\ \ {}_{0^{+}}^{C}\beth_{\mu}^{\alpha}\left[\mu(t)\right]^{\beta-1}=
  \frac{\Gamma(\beta)}{\Gamma(\beta-\alpha)}\left[\mu(t)\right]^{\beta-\alpha-1}
\end{equation*}

\section{Main Results}

In the current section, we present our findings to begin with defining  the $\mu$-neutral multi-delayed perturbation of Mittag-Leffler type matrix function. We look for  a solution of system (\ref{sorkok}) and prove the existence and uniqueness of solutions and Ulam-Hyers stability of system (\ref{sorkok}).

\subsection{The $\mu$- neutral multi-delayed perturbation of Mittag-Leffler type matrix function}

For now on,  we exploit  the $\mu$-{\footnotesize NMDP} of {\footnotesize ML} function for the $\mu$-neutral multi-delayed perturbation of Mittag-Leffler type matrix function.

It is clear that a generalisation of the exponential function is the {\footnotesize ML} function.  Delayed and Delayed perturbed and  multi-delayed perturbed  versions which are called delayed  Mittag-Leffler type matrix function, delayed perturbation of Mittag-Leffler type matrix function and  multi-delayed perturbation of Mittag-Leffler type matrix function by the numbers are  presented in the work \cite{i3}\cite{i5}\cite{1} respectively. In the current study,  the $\mu$-{\footnotesize NMDP} of {\footnotesize ML} function is given through identifying matrix equation for $Q_j\left(  s\right)$ for $j=0,1,2, \dots$
\begin{align}\label{qkme}
 Q_{j+1}\left(  s_1, s_2, \dots, s_d\right)  & = B Q_{j}\left(  s_1, s_2, \dots, s_d\right)
 + \sum_{i=1}^{d}F_{i}Q_{j}\left(  s_1, s_2, \dots, s_{i}-r_{i}, \dots, s_d\right)\nonumber\\
  &+ \sum_{i=1}^{d}A_{i}Q_{j+1}\left(  s_1, s_2, \dots, s_{i}-r_{i}, \dots, s_d\right),
 \end{align}
$$Q_{0}\left(  s_1, s_2, \dots, s_d\right)=Q_{j}\left(  -r_1,  \dots, s_d\right)=Q_{j}\left(  s_1,  \dots, -r_d\right) =\Theta, $$
 $$Q_{1}\left(  0,\dots,0\right)= I, \ \  \\ Q_{1}\left( s_1, s_2, \dots, s_d\right)= \Theta, \ s_i\neq 0. $$
where $s_{i}=0,r_{i}, 2r_{i}, \dots$, $\Theta$ is the zero matrix, and $I$ is the unit matrix.

In the following definition, we give   the function of {$\mu$-\footnotesize NMDP} of {\footnotesize ML} type matrix  with the aid of  the multivariate function $Q_{k+1}\left(  s_1, s_2, \dots, s_d\right)$
\begin{definition}
The $\mu$-neutral multi-delayed perturbation of the Mittag-Leffler type matrix function $ {\mathcal{X}_\mu^{\alpha,\beta}} (t,s)$  is given by

\begin{equation}\label{mainm}
   {\mathcal{X}_\mu^{\alpha,\beta}} (t,s)= \left\{\begin{array}{cc}
                                           \Theta, \ & -r\leq t < 0, \ s\geq 0, \\
  {\displaystyle\sum\limits_{k=0}^{\infty}}{\displaystyle\sum\limits_{i_1,i_2,\dots,i_d=0}^{\infty}}\mathcal{Q}_{k+1}
   \frac{\left[\mu\left(t\right)-\mu\left(s+ {\sum_{j=1}^{d}} i_jr_j\right)\right]_+^{k\alpha+\beta-1}}{\Gamma(k\alpha+\beta)}, & t\geq 0, \  s\geq0,
                                                 \end{array}\right.
\end{equation}
where $\mathcal{Q}_{k+1}:=Q_{k+1}(i_1r_1,\dots, i_dr_d)$, $[t]_+=\max(0,t)$.
\end{definition}

\begin{remark}
Let $\mu(t)=t$, and $s=0$ in  $ {\mathcal{X}^{\alpha,\beta}_\mu}(t,s)$. Then we have,
\begin{enumerate}
\item For $A_{i}=F_{i}=\Theta$, $i=1,2, \dots, d$,   The {$\mu$-\footnotesize NMDP}  {\footnotesize ML}  function reduces to the {\footnotesize ML} function\cite{4}   i.e. $ {\mathcal{X}^{\alpha,\beta}_\mu}(t,s) =t^{\beta-1}E_{\alpha, \beta}\left( Bt^{\alpha} \right)$.
\item For  $A_{i}=\Theta$, $i=1,2, \dots, d$ and $F_{i}=\Theta$, $i=2, \dots, d$, then the {$\mu$-\footnotesize NMDP}  {\footnotesize ML}  function matches up with delay perturbation of {\footnotesize ML} function\cite{i3}.
\item For $A_{i}=\Theta$, $i=1,2, \dots, d$ and $F_{i}=\Theta$, $i=2, \dots, d$, $B=\Theta$,  $ {\mathcal{X}^{\alpha,\beta}_\mu}(t,s)$ reduces to delayed {\footnotesize ML} function\cite{1} .
\item  The {$\mu$-\footnotesize NMDP}  {\footnotesize ML} function is not equal to that one of \cite[Def. 3.3]{i5} because the matrix  $Q_k\left(  s\right)$ for $k=0,1,2, \dots$ in (\ref{qkme}) is  different from that  of the study \cite{i5}. Under the condition $A_{i}=\Theta$, $i=1,2, \dots, d$, however,  they coincide.
\item     Since the constant matrices are commutative together with suitable selections, ${\mathcal{X}^{\alpha,\beta}_\mu}(t,s)$ overlaps with $X(t)$ in (2.4) in the work\cite{6}.
\item By depending on suitable selections $r_i, A_i, F_i, B$, $i=1,2,...,d$, one can easily obtain $\varepsilon_{\alpha,\beta}^{r_1,r_2}(\mathcal{A},\mathcal{B},\mathcal{F};x)$ in Definition 3.1 in the work\cite{m10} from $ {\mathcal{X}^{\alpha,\beta}_\mu}(t,s)$.
\end{enumerate}

\end{remark}

\subsection{The $\mu$-neutral multi-delayed Caputo fractional differential equations' analytic solutions }

First of all, we share beneficial theorem and lemma to be exploited in the next proofs.

\begin{theorem}\label{lemmaessential}
Let ${\mathcal{X}^{\alpha,\beta}_\mu}(t,s) $ be as defined in (\ref{mainm}). The following holds true
\begin{equation*}
  {_{0^+}^{{C}}\beth^{\alpha}_{\mu}}\left[  {\mathcal{X}^{\alpha,1}_\mu} \left( t,0 \right) -\sum_{j=1}^{d}A_{j}     {\mathcal{X}^{\alpha,1}_\mu} \left( t, r_{j} \right)\right]= B {\mathcal{X}^{\alpha,1}_\mu} \left( t,0 \right)+  \sum_{j=1}^{d}F_{j}  {\mathcal{X}^{\alpha,1}_\mu}\left( t,r_j\right)
\end{equation*}
\end{theorem}

\begin{proof}
   Firstly, we  compute the term ${_{0^+}^{{C}}\beth^{\alpha}_{\mu}}{\mathcal{X}^{\alpha,1}_\mu}\left( t,0 \right)$.
\begin{align}\label{ilk1}
&  {_{0^+}^{{C}}\beth^{\alpha}_{\mu}}{\mathcal{X}^{\alpha,1}_\mu}\left( t,0 \right)  \nonumber \\
  &= {\displaystyle\sum\limits_{k=1}^{\infty}}{\displaystyle\sum\limits_{i_1,i_2,\dots,i_d=0}^{\infty}}Q_{k+1}(i_1r_1,\dots, i_dr_d){_{0^+}^{{C}}\beth^{\alpha}_{\mu}}\left(
   \frac{\left[\mu(t)-\mu\left(\sum_{n=1}^{d}i_nr_n\right)\right]_+^{k\alpha}}{\Gamma(k\alpha+1)} \right) \nonumber \\
  & ={\displaystyle\sum\limits_{k=0}^{\infty}}{\displaystyle\sum\limits_{i_1,i_2,\dots,i_d=0}^{\infty}}Q_{k+2}(i_1r_1,\dots, i_dr_d)
   \frac{\left[\mu(t)-\mu\left(\sum_{n=1}^{d}i_nr_n\right)\right]_+^{k\alpha}}{\Gamma(k\alpha+1)}.
   \end{align}
One can obtain ${_{0^+}^{{C}}\beth^{\alpha}_{\mu}}{\mathcal{X}^{\alpha,1}_\mu}\left( t,r_j \right)$ as follows  by using the similar calculations.
  \begin{equation*}
    {_{0^+}^{{C}}\beth^{\alpha}_{\mu}}{\mathcal{P}^\mu_{\alpha,1}}\left( t,r_j \right):={\displaystyle\sum\limits_{k=0}^{\infty}}{\displaystyle\sum\limits_{i_1,i_2,\dots,i_d=0}^{\infty}}Q_{k+2}(i_1r_1,\dots, i_dr_d)
   \frac{\left[\mu(t)-\mu\left(r_j+\sum_{n=1}^{d}i_nr_n\right)\right]_+^{k\alpha}}{\Gamma(k\alpha+1)}.
  \end{equation*}
 By combining  (\ref{qkme}) with (\ref{ilk1}), we keep calculating
 \begin{align*}
& {_{0^+}^{{C}}\beth^{\alpha}_{\mu}}{\mathcal{X}^{\alpha,1}_\mu}\left( t,0 \right)\\
  & =B {\displaystyle\sum\limits_{k=0}^{\infty}}{\displaystyle\sum\limits_{i_1,i_2,\dots,i_d=0}^{\infty}}Q_{k+1}(i_1r_1,\dots, i_dr_d)
   \frac{\left[\mu(t)-\mu\left(\sum_{n=1}^{d}i_nr_n\right)\right]_+^{k\alpha}}{\Gamma(k\alpha+1)} \\
  & +\sum_{j=1}^{d}F_j {\displaystyle\sum\limits_{k=0}^{\infty}}{\displaystyle\sum\limits_{i_1,i_2,\dots,i_d=0}^{\infty}}Q_{k+1}(i_1r_1,\dots,i_jr_j -r_j,\dots,  i_dr_d)
   \frac{\left[\mu(t)-\mu\left(\sum_{n=1}^{d}i_nr_n\right)\right]_+^{k\alpha}}{\Gamma(k\alpha+1)} \\
  & +\sum_{j=1}^{d}A_j {\displaystyle\sum\limits_{k=0}^{\infty}}{\displaystyle\sum\limits_{i_1,i_2,\dots,i_d=0}^{\infty}}Q_{k+2}(i_1r_1,\dots,i_jr_j -r_j,\dots,  i_dr_d)
   \frac{\left[\mu(t)-\mu\left(\sum_{n=1}^{d}i_nr_n\right)\right]_+^{k\alpha}}{\Gamma(k\alpha+1)} \\
   & =B {\displaystyle\sum\limits_{k=0}^{\infty}}{\displaystyle\sum\limits_{i_1,i_2,\dots,i_d=0}^{\infty}}Q_{k+1}(i_1r_1,\dots,  i_dr_d)
   \frac{\left[\mu(t)-\mu\left(\sum_{n=1}^{d}i_nr_n\right)\right]_+^{k\alpha}}{\Gamma(k\alpha+1)} \\
  & +\sum_{j=1}^{d}F_j {\displaystyle\sum\limits_{k=0}^{\infty}}{\displaystyle\sum\limits_{i_1,i_2,\dots,i_d=0}^{\infty}}Q_{k+1}(i_1r_1,\dots,  i_dr_d)
   \frac{\left[\mu(t)-\mu\left(r_j+\sum_{n=1}^{d}i_nr_n\right)\right]_+^{k\alpha}}{\Gamma(k\alpha+1)} \\
  & +\sum_{j=1}^{d}A_j {\displaystyle\sum\limits_{k=0}^{\infty}}{\displaystyle\sum\limits_{i_1,i_2,\dots,i_d=0}^{\infty}}Q_{k+2}(i_1r_1,\dots,  i_dr_d)
   \frac{\left[\mu(t)-\mu\left(r_j+\sum_{n=1}^{d}i_nr_n\right)\right]_+^{k\alpha}}{\Gamma(k\alpha+1)} \\
   &= B {\mathcal{X}^{\alpha,1}_\mu} \left( t,0 \right)+  \sum_{j=1}^{d}F_{j}  {\mathcal{X}^{\alpha,1}_\mu}\left( t,r_j\right) +\sum_{k=1}^{d}A_j \left( {_{0^+}^{{C}}\beth^{\alpha}_{\mu}}{\mathcal{X}^{\alpha,1}_\mu}\left( t,r_j \right)\right).
\end{align*}
 which provides the craved result.

\end{proof}

\begin{corollary}
A solution of system (\ref{sorkok}) is
\begin{equation*}
    w\left(t  \right)=\left[ {\mathcal{X}^{\alpha,1}_\mu} \left( t,0 \right) -\sum_{j=1}^{d}     {\mathcal{X}^{\alpha,1}_\mu} \left( t, r_{j} \right)A_{j} \right]\varphi(0),
\end{equation*}
 provided that $ w\left(t  \right)=\varphi\left(t  \right)$ with $-r\leq t\leq0$.
\end{corollary}
\begin{proof}
 With the help of the variation of constants' technique, set
\begin{equation*}
    w\left(t  \right)=\left[ {\mathcal{X}^{\alpha,1}_\mu} \left( t,0 \right) -\sum_{j=1}^{d}     {\mathcal{X}^{\alpha,1}_\mu} \left( t, r_{j} \right)A_{j} \right]\kappa,
\end{equation*}
where a constant  $\kappa$ is unknown. It is clear that ${\mathcal{X}^{\alpha,1}_\mu} \left( 0,0 \right)=I$ and  ${\mathcal{X}^{\alpha,1}_\mu} \left(0, r_j \right)=\Theta$. As a result, $\kappa=\varphi(0)$.
\end{proof}

\begin{lemma}\label{temelremark2}
Let  ${\mathcal{X}^{\alpha,\beta}_\mu}\left( t,s \right) $ be as in (\ref{mainm}). The following mathematical equation is true:
\begin{align*}
 & \int_{0}^{t}\left( \mu(t)-\mu(s) \right)^{-\alpha}\int_{0}^{s}{\mathcal{X}^{\alpha,\alpha}_{\mu}}\left( s,x \right) \daleth\left(  x\right)d\mu(x)d\mu(s)\\
  &= {\displaystyle\sum\limits_{k=0}^{\infty}}{\displaystyle\sum\limits_{i_1,i_2,\dots,i_d=0}^{\infty}}Q_{k+1}(i_1r_1,\dots, i_dr_d) \int_{0}^{t}
   \frac{\Gamma(1-\alpha)\left[\mu(t)-\mu(x+\sum_{n=i}^{d}i_nr_n)\right]_+^{k\alpha}}{\Gamma(k\alpha+1)} \daleth\left(  x\right)d\mu(x).
\end{align*}
\end{lemma}
\begin{proof}
 One can easily prove this theorem with the help of  a simple substitution $v=\frac{\mu(s)-\mu(x+\sum_{n=1}^{d}i_nr_n)}{\mu(t)-\mu(x+\sum_{n=1}^{d}i_nr_n)}$ together with the expansion of ${\mathcal{X}^{\alpha,\alpha}_{\mu}}$, so it is ignored.
\end{proof}

Now, the coming theorem is one of main theorems as to the desired solutions. It gives a part of the solution under the zero initial condition .

\begin{theorem}\label{thmessential}
The following function
\begin{equation*}
    w\left(t  \right)=\int_{0}^{t}{\mathcal{X}^{\alpha,\alpha}_{\mu}}\left(t,s  \right)\daleth\left(  s\right)d\mu(s), \ \ \ t\geq 0,
\end{equation*}
is a solution of system (\ref{sorkok}) under the condition $ w\left(t  \right)=0$ with $-r\leq t\leq0$.
\end{theorem}

\begin{proof}
  To see this, we consider the following expression by keeping Lemma \ref{temelremark2} in mind
  \begin{align}\label{ilk2}
   &{_{0^+}^{{C}}\beth^{\alpha}_{\mu}}\left( \int_{0}^{t}{\mathcal{X}^{\alpha,\alpha}_\mu}\left(t,x  \right)\daleth\left(  x\right)d\mu(x) \right)   \nonumber\\
   &=\frac{1}{\Gamma(1-\alpha)} \frac{d}{d\mu(t)} \int_{0}^{t}\left( \mu(t)-\mu(s) \right)^{-\alpha}\int_{0}^{s}{\mathcal{X}^{\alpha,\alpha}_{\mu}}\left( s,x \right) \daleth\left(  x\right)d\mu(x)d\mu(s)  \nonumber\\
     &  =   {\displaystyle\sum\limits_{k=0}^{\infty}}{\displaystyle\sum\limits_{i_1,i_2,\dots,i_d=0}^{\infty}}Q_{k+1}(i_1r_1,\dots, i_dr_d) \nonumber \\
  &\times \frac{d}{d\mu(t)}\int_{0}^{t}
   \frac{\left[\mu(t)-\mu(x+\sum_{n=1}^{d}i_nr_n)\right]_+^{k\alpha}}{\Gamma(k\alpha+1)} \daleth\left(  x\right)d\mu(x)  \nonumber\\
    &  =   {\displaystyle\sum\limits_{k=0}^{\infty}}{\displaystyle\sum\limits_{i_1,i_2,\dots,i_d=0}^{\infty}}Q_{k+1}(i_1r_1,\dots, i_dr_d) \nonumber \\
  &\times \frac{d}{d\mu(t)}\int_{0}^{t-\sum_{n=1}^{d}i_nr_n}
   \frac{\left[\mu(t)-\mu(x+\sum_{n=1}^{d}i_nr_n)\right]_+^{k\alpha}}{\Gamma(k\alpha+1)} \daleth\left(  x\right)d\mu(x) \nonumber  \nonumber\\
    &  =   {\displaystyle\sum\limits_{k=1}^{\infty}}{\displaystyle\sum\limits_{i_1,i_2,\dots,i_d=0}^{\infty}}Q_{k+1}(i_1r_1,\dots, i_dr_d) \nonumber \\
  &\times \frac{d}{d\mu(t)}\int_{0}^{t-\sum_{n=1}^{d}i_nr_n}
   \frac{\left[\mu(t)-\mu(x+\sum_{n=1}^{d}i_nr_n)\right]_+^{k\alpha}}{\Gamma(k\alpha+1)} \daleth\left(  x\right)d\mu(x)+\daleth\left(  t\right) \nonumber \\
    &  =   {\displaystyle\sum\limits_{k=0}^{\infty}}{\displaystyle\sum\limits_{i_1,i_2,\dots,i_d=0}^{\infty}}Q_{k+2}(i_1r_1,\dots, i_dr_d)\nonumber \\
  &\times \int_{0}^{t}
   \frac{\left[\mu(t)-\mu(x+\sum_{n=1}^{d}i_nr_n)\right]_+^{k\alpha+\alpha-1}}{\Gamma(k\alpha+\alpha)} \daleth\left(  x\right)d\mu(x)+\daleth\left(  t\right).
  \end{align}
  One can easily obtain
  \begin{align*}
  &{_{0^+}^{{C}}\beth^{\alpha}_{\mu}}\left( \int_{0}^{t}{\mathcal{X}^{\alpha,\alpha}_{\mu}}\left(t,x+r_j  \right)\daleth\left(  x\right)d\mu(x) \right)=  {\displaystyle\sum\limits_{k=0}^{\infty}}{\displaystyle\sum\limits_{i_1,i_2,\dots,i_d=0}^{\infty}}Q_{k+2}(i_1r_1,\dots, i_dr_d)  \\ &\times \int_{0}^{t}
   \frac{\left[\mu(t)-\mu(x+r_j+\sum_{n=1}^{d}i_nr_n)\right]_+^{k\alpha+\alpha-1}}{\Gamma(k\alpha+\alpha)} \daleth\left(  x\right)d\mu(x)
  \end{align*}
By having Lemma \ref{temelremark2} in mind, we combine (\ref{qkme}) with  (\ref{ilk2})  to obtain the following
\begin{align*}
  &{_{0^+}^{{C}}\beth^{\alpha}_{\mu}}\left( \int_{0}^{t}{\mathcal{X}^{\alpha,\alpha}_{\mu}}\left(t,x  \right)\daleth\left(  x\right)d\mu(x) \right) \\
  & = B {\displaystyle\sum\limits_{k=0}^{\infty}}{\displaystyle\sum\limits_{i_1,i_2,\dots,i_d=0}^{\infty}}Q_{k+1}(i_1r_1,\dots, i_dr_d) \\
 & \times \int_{0}^{t}
   \frac{\left[\mu(t)-\mu(x+\sum_{n=1}^{d}i_nr_n)\right]_+^{k\alpha+\alpha-1}}{\Gamma(k\alpha+\alpha)} \daleth\left(  x\right)d\mu(x)  +\daleth\left(  t\right)\\
   & + \sum_{j=1}^{d} F_j {\displaystyle\sum\limits_{k=0}^{\infty}}{\displaystyle\sum\limits_{i_1,i_2,\dots,i_d=0}^{\infty}}Q_{k+1}(i_1r_1,\dots,i_jr_j-r_j,\dots, i_dr_d)  \\
 & \times  \int_{0}^{t}
   \frac{\left[\mu(t)-\mu(x+\sum_{n=1}^{d}i_nr_n)\right]_+^{k\alpha+\alpha-1}}{\Gamma(k\alpha+\alpha)} \daleth\left(  x\right)d\mu(x)\\
   & + \sum_{j=1}^{d} A_j  {\displaystyle\sum\limits_{k=0}^{\infty}}{\displaystyle\sum\limits_{i_1,i_2,\dots,i_d=0}^{\infty}}Q_{k+2}(i_1r_1,\dots,i_jr_j-r_j,\dots, i_dr_d)   \\
 & \times \int_{0}^{t}
   \frac{\left[\mu(t)-\mu(x+\sum_{n=1}^{d}i_nr_n)\right]_+^{k\alpha+\alpha-1}}{\Gamma(k\alpha+\alpha)} \daleth\left(  x\right)d\mu(x) \\
   & = B{\displaystyle\sum\limits_{k=0}^{\infty}}{\displaystyle\sum\limits_{i_1,i_2,\dots,i_d=0}^{\infty}}Q_{k+1}(i_1r_1,\dots, i_dr_d)  \\
 & \times  \int_{0}^{t}
   \frac{\left[\mu(t)-\mu(x+\sum_{n=1}^{d}i_nr_n)\right]_+^{k\alpha+\alpha-1}}{\Gamma(k\alpha+\alpha)} \daleth\left(  x\right)d\mu(x)  +\daleth\left(  t\right)\\
   & + \sum_{j=1}^{d} F_j {\displaystyle\sum\limits_{k=0}^{\infty}}{\displaystyle\sum\limits_{i_1,i_2,\dots,i_d=0}^{\infty}}Q_{k+1}(i_1r_1,\dots, i_dr_d)  \\
 & \times  \int_{0}^{t}
   \frac{\left[\mu(t)-\mu(t+r_j+\sum_{n=1}^{d}i_nr_n)\right]_+^{k\alpha+\alpha-1}}{\Gamma(k\alpha+\alpha)} \daleth\left(  x\right)d\mu(x)\\
   & + \sum_{j=1}^{d} A_j  {\displaystyle\sum\limits_{k=0}^{\infty}}{\displaystyle\sum\limits_{i_1,i_2,\dots,i_d=0}^{\infty}}Q_{k+2}(i_1r_1,\dots, i_dr_d)  \\
 & \times  \int_{0}^{t}
   \frac{\left[\mu(t)-\mu(x+r_j+\sum_{n=1}^{d}i_nr_n)\right]_+^{k\alpha+\alpha-1}}{\Gamma(k\alpha+\alpha)} \daleth\left(  x\right)d\mu(x) \\
   &=B \int_{0}^{t}{\mathcal{X}^{\alpha,\alpha}_\mu}\left(t,x  \right)\daleth\left(  x\right)d\mu(x)  +    \sum_{j=1}^{d} F_j \int_{0}^{t}{\mathcal{X}^{\alpha,\alpha}_\mu}\left(t,x+j_k  \right)\daleth\left(  x\right)d\mu(x)  + \daleth\left(  t\right)   \\
         &+ \sum_{j=1}^{d} A_j \left[ {_{0^+}^{{C}}\beth^{\alpha}_{\mu}}\left( \int_{0}^{t}{\mathcal{X}^{\alpha,\alpha}_\mu}\left(t,x+r_j  \right)\daleth\left(  x\right)d\mu(x) \right)\right]\\
\end{align*}
which gives the inevitable result.
\end{proof}

The next theorem is the last one of main theorems as to a solution of the homogeneous part of system (\ref{sorkok}).

\begin{theorem}\label{thmessential2}
The following $\mathbb{R}^n$-valued continuous function
\begin{equation*}
    w\left(t  \right)=\sum_{j=1}^{d}\int_{-r_j}^{0} {\mathcal{X}^{\alpha,\alpha}_\mu}\left( t,r_j+s \right) \left[F_j\varphi\left( s \right)+A_j \left({_{0^+}^{{C}}\beth^{\alpha}_{\mu}}\varphi \right) \left( s \right) \right] d\mu(s)
\end{equation*}
 is a solution of system (\ref{sorkok})  with $ w\left(t  \right)=\varphi\left(t  \right)$, $-r\leq t\leq0$  and $\daleth=0$.
\end{theorem}

\begin{proof} Now we consider
  \begin{align}\label{ilk3}
   &  {_{0^+}^{{C}}\beth^{\alpha}_{\mu}}\left( \sum_{j=1}^{d}\int_{-r_j}^{0} {\mathcal{X}^{\alpha,\alpha}_\mu}\left( t,r_j+x \right) F_j\phi\left(x \right)d\mu(x)  \right) \nonumber \\
   & =  {_{0^+}^{{C}}\beth^{\alpha}_{\mu}} \biggl( \sum_{j=1}^{d}\int_{-r_j}^{0}  {\displaystyle\sum\limits_{k=1}^{\infty}}{\displaystyle\sum\limits_{i_1,i_2,\dots,i_d=0}^{\infty}}Q_{k+1}(i_1r_1,\dots, i_dr_d) \nonumber \\
 & \times  \frac{\left[\mu(t)-\mu(x+r_j+\sum_{n=1}^{d} i_nr_n)\right]_+^{k\alpha+\alpha-1}}{\Gamma(k\alpha+\alpha)}F_j\varphi\left( x \right)d\mu(x) \biggl) \nonumber\\
   & =  {_{0^+}^{{C}}\beth^{\alpha}_{\mu}} \biggl( \sum_{j=1}^{d}\int_{-r_j}^{0}  {\displaystyle\sum\limits_{k=0}^{\infty}}{\displaystyle\sum\limits_{i_1,i_2,\dots,i_d=0}^{\infty}}Q_{k+2}(i_1r_1,\dots, i_dr_d)  \nonumber \\
 & \times
   \frac{\left[\mu(t)-\mu(x+r_j+\sum_{n=1}^{d} i_nr_n)\right]_+^{k\alpha+2\alpha-1}}{\Gamma(k\alpha+2\alpha)} F_j\varphi\left( x \right)d\mu(x) \biggl)
  \end{align}
By applying (\ref{qkme}) to (\ref{ilk3}), we get
\begin{align*}
   &  {_{0^+}^{{C}}\beth^{\alpha}_{\mu}}\left( \sum_{j=1}^{d}\int_{-r_j}^{0} {\mathcal{X}^{\alpha,\alpha}_\mu}\left( t,x+r_j \right) F_j\varphi\left( x \right)d\mu(x)  \right) \\
   & = B\sum_{j=1}^{d}\int_{-r_j}^{0}  {\displaystyle\sum\limits_{k=0}^{\infty}}{\displaystyle\sum\limits_{i_1,i_2,\dots,i_d=0}^{\infty}}Q_{k+1}(i_1r_1,\dots, i_dr_d) \\
 & \times  {_{0^+}^{{C}}\beth^{\alpha}_{\mu}}\frac{\left[\mu(t)-\mu(t+r_j+\sum_{n=1}^{d} i_nr_n )\right]_+^{k\alpha+2\alpha-1}}{\Gamma(k\alpha+2\alpha)} F_j\varphi\left( x \right)d\mu(x) \\
   & +\sum_{m=1}^{d}C_m \sum_{j=1}^{d}\int_{-r_j}^{0}  {\displaystyle\sum\limits_{k=0}^{\infty}}{\displaystyle\sum\limits_{i_1,i_2,\dots,i_d=0}^{\infty}}Q_{k+1}(i_1r_1,\dots,i_mr_m-r_m,\dots, i_dr_d) \\
 & \times  {_{0^+}^{{C}}\beth^{\alpha}_{\mu}}\frac{\left[\mu(t)-\mu(x+r_j+\sum_{n=1}^{d} i_nr_n\right]_+^{k\alpha+2\alpha-1}}{\Gamma(k\alpha+2\alpha)} F_j\varphi\left( x \right)d\mu(x) \\
   &  +\sum_{m=1}^{d}A_m \biggl[{_{0^+}^{{C}}\beth^{\alpha}_{\mu}}\bigg( \sum_{j=1}^{d}\int_{-r_j}^{0}  {\displaystyle\sum\limits_{k=0}^{\infty}}{\displaystyle\sum\limits_{i_1,i_2,\dots,i_d=0}^{\infty}}Q_{k+2}(i_1r_1,\dots,i_mr_m-r_m,\dots, i_dr_d) \\
 & \times   \frac{\left[\mu(t)-\mu(x+r_j+\sum_{n=1}^{d} i_nr_n)\right]_+^{k\alpha+2\alpha-1}}{\Gamma(k\alpha+2\alpha)} F_j\varphi\left( x \right)d\mu(x)\bigg) \biggl] \\
   & = B\sum_{j=1}^{d}\int_{-r_j}^{0}  {\displaystyle\sum\limits_{k=0}^{\infty}}{\displaystyle\sum\limits_{i_1,i_2,\dots,i_d=0}^{\infty}}Q_{k+1}(i_1r_1,\dots, i_dr_d)  \\
 & \times
    \frac{\left[\mu(t)-\mu(x+r_j\sum_{n=1}^{d} i_nr_n)\right]_+^{k\alpha+\alpha-1}}{\Gamma(k\alpha+\alpha)} F_j\varphi\left( x \right)d\mu(x) \\
   & +\sum_{m=1}^{d}C_m \sum_{j=1}^{d}\int_{-r_j}^{0}  {\displaystyle\sum\limits_{k=0}^{\infty}}{\displaystyle\sum\limits_{i_1,i_2,\dots,i_d=0}^{\infty}}Q_{k+1}(i_1r_1,\dots, i_dr_d)  \\
 & \times
    \frac{\left[\mu(t)-\mu(x+r_j+r_m+\sum_{n=1}^{d} i_nr_n)\right]_+^{k\alpha+\alpha-1}}{\Gamma(k\alpha+\alpha)} F_j\varphi\left( x \right)d\mu(x) \\
   &  +\sum_{m=1}^{d}A_m \biggl[{_{0^+}^{{C}}\beth^{\alpha}_{\mu}}\bigg( \sum_{j=1}^{d}\int_{-r_j}^{0}  {\displaystyle\sum\limits_{k=0}^{\infty}}{\displaystyle\sum\limits_{i_1,i_2,\dots,i_d=0}^{\infty}}Q_{k+1}(i_1r_1,\dots, i_dr_d)   \\
 & \times
   \frac{\left[\mu(t)-\mu(x+r_j+r_m+\sum_{n=1}^{d} i_nr_n)\right]_+^{k\alpha+\alpha-1}}{\Gamma(k\alpha+\alpha)} F_j\varphi\left( x \right)d\mu(x)\bigg) \biggl] \\
   &= B \left( \sum_{j=1}^{d}\int_{-r_j}^{0} {\mathcal{X}^{\alpha,\alpha}_\mu}\left( t,x+r_j \right) F_j\varphi\left( x \right)d\mu(x)  \right) \\ & +\sum_{m=1}^{d}C_m\left( \sum_{j=1}^{d}\int_{-r_j}^{0} {\mathcal{P}_{\alpha,\alpha}}\left( t,x+r_j+r_m \right) F_j\varphi\left( x \right)d\mu(x)  \right) \\
   & +\sum_{m=1}^{d}A_m \left[{_{0^+}^{{C}}\beth^{\alpha}_{\mu}}\left( \sum_{j=1}^{d}\int_{-r_j}^{0} {\mathcal{X}{\alpha,\alpha}^\mu}\left( t, x+r_j+r_m \right) F_j\varphi\left( x \right)d\mu(x)  \right)\right].
\end{align*}
which provides $\sum_{j=1}^{d}\int_{-r_j}^{0} {\mathcal{X}^{\alpha,\alpha}_\mu}\left( t,x+r_j \right) A_j \left({_{0^+}^{{C}}\beth^{\alpha}_{\mu}}\varphi\right)\left( x \right)d\mu(x)$.  In a similar way, it can  be easily shown that $\sum_{j=1}^{d}\int_{-r_j}^{0} {\mathcal{X}^{\alpha,\alpha}_\mu}\left( t,x+r_j \right) F_j\varphi\left( x \right)d\mu(x)$ is a solution. Furthermore, summation of them is also  a solution due to superposition technique. This completes the proof.
\end{proof}

So far we have found the parts of the stepwise solution, now let's put the parts together in the below corollary which stands for a whole solution of system (\ref{sorkok}).

\begin{corollary}\label{soncor}
The following $\mathbb{R}^n$-valued continuous function
\begin{align*}
    w\left(t  \right)&=\left[ {\mathcal{X}^{\alpha,1}_\mu} \left( t,0 \right)- \sum_{m=1}^{d} {\mathcal{X}^{\alpha,1}_\mu} \left( t,r_m \right)A_{m}   \right]\varphi(0)+ \int_{0}^{t}{\mathcal{X}^{\alpha,\alpha}_{\mu}}\left(t,s  \right)\daleth\left(  s\right)d\mu(s) \\
    & +\sum_{m=1}^{d}\int_{-r_m}^{0} {\mathcal{X}^{\alpha,\alpha}_{\mu}}\left( t,s+r_m \right) \left[F_m\varphi\left( s \right)+A_m \left({_{0^+}^{{C}}\beth^{\alpha}_{\mu}}\varphi\right)\left( s \right) \right] d\mu(s)
\end{align*}
 is an exact analytical solution of system (\ref{sorkok}). In the light of Corollary \ref{soncor}, one can easily acquire  a mild Volterra type integral equation's solution formula of system (\ref{sorkok2}).
\end{corollary}

\begin{remark}
Here are some special cases depending on selections of the coefficient matrices. For $\mu(x)=x$ and $y=0$
\begin{enumerate}
\item If  $A_{j}=\Theta$, $j=1,2, \dots, d$ and $F_{j}=\Theta$, $j=2, \dots, d$, then Corollary \ref{soncor} matches up with  Corollary $1$ in the reference\cite{i3}.
\item Corollary \ref{soncor} with $\daleth=0$ reduces to Theorem 3.2 in the work\cite{1} providing that $a_{j}=\Theta$, $j=1,2, \dots, d$ and $F_{j}=\Theta$, $j=2, \dots, d$, $B=\Theta$ .
\item Corollary \ref{soncor} overlap with Theorem 4.2 in the study\cite{i5} under the condition $A_{j}=\Theta$, $j=1,2, \dots, d$.
\item  Even if the constant coefficient matrices are commutative, our findings also are valid. If the coefficient matrices are permutable in addition to appropriate selections, Corollary  \ref{soncor} reduces to Theorem 6 in the paper\cite{6}.
\item Corollary \ref{soncor} reduces to Theorem 3.5 in the paper\cite{m10} on taking $d=2$ and without loss of generality $A_1=F_2=\Theta$.
\end{enumerate}

\end{remark}

\subsection{Existence and  uniqueness of solutions of neutral fractional multi-delayed differential equations' system}
 In this subsection, we look for answers to three kinds of questions : is there a solution for system (\ref{sorkok2})?, is the solution unique? Subsequent to given answers, we finis discussing. When we look at features of each term in system (\ref{sorkok2}) like $\daleth\left(  t,w\left(  t\right)\right)$ is continuous, we find an explicit solution in corollary \ref{soncor}. Unfortunately, these features or conditions are not enough to make the explicit solution unique. So, we add one more feature to the continuous function $\daleth\left(  t,w\left(  t\right)\right)$ in order to make the explicit solution  satisfy the uniqueness. This feature is that the continuous function $\daleth\left(  t,w\left(  t\right)\right)$ satisfies the Lipschitz condition in the second component with the Lipschitz constant $L_{\daleth}$, that is
 \begin{equation*}
   \left\Vert \daleth\left(  t,w\left(  t\right)\right)- \daleth\left(  t,u\left(  t\right)\right)  \right\Vert \leq  L_{\daleth}  \left\Vert w\left(  t\right)-u\left(  t\right) \right\Vert.
 \end{equation*}

Prior to carrying on, we discuss an inequality about ${\mathcal{X}^{\alpha,\beta}_\mu}$ in the following lemma.

\begin{lemma}\label{eql}
Let $\mathcal{X}^{\alpha,\beta}_\mu\left(  t,s\right)$ be as in (\ref{mainm}).
\begin{equation*}
  \int_{0}^{t} \left \Vert\mathcal{X}^{\alpha,\beta}_\mu \left(  t,s\right) \right\Vert d\mu(s) \leq \left(\mu(t) - \mu(0)\right)\left \Vert\mathcal{X}^{\alpha,\beta}_\mu \left(  t,0\right) \right\Vert
\end{equation*}
holds true.
\end{lemma}

\begin{proof}
 In order to see whether the above inequality is true, it is sufficient to use   $\left[\mu(t)-\mu\left(s+ \sum_{j=1}^{d}  i_jr_j\right)\right] \leq \left[\mu(t)-\mu\left( \sum_{j=1}^{d}  i_jr_j\right)\right]$ along with the expansion $\mathcal{X}^{\alpha,\beta}_\mu \left(  t,s\right)$. So, it is omitted.
\end{proof}

The following theorem is about  existence and uniqueness of system (\ref{sorkok2}).
\begin{theorem}\label{eu1}
  If the continuous function $\daleth\left(  t,w\left(  t\right)\right)$ satisfies the Lipschitz condition in the second component with the Lipschitz constant $L_{\daleth}$ with $    L_{\daleth} \left(\mu(T) - \mu(0)\right) \left \Vert\mathcal{X}^{\alpha,\alpha}_\mu \left(  T,0\right) \right\Vert <1$, then the integral equation in the corollary \ref{soncor} is of a unique solution in  $\left[-d, T  \right]$.
\end{theorem}

\begin{proof}
  Define $\mathcal{G}: C\left( \left[-r,T  \right], \mathbb{R}^n \right) \rightarrow C\left( \left[-r,T  \right], \mathbb{R}^n \right) $ by
\begin{align*}
  \mathcal{G} w\left(t  \right)&=\left[ {\mathcal{X}^{\alpha,1}_\mu} \left( t,0 \right)- \sum_{m=1}^{d} {\mathcal{X}^{\alpha,1}_\mu} \left( t,r_m \right)A_{m}   \right]\varphi(0) + \int_{0}^{t}{\mathcal{X}^{\alpha,\alpha}_{\mu}}\left(t,x  \right)\daleth\left(  x,w(x)\right)d\mu(x) \\
    & +\sum_{m=1}^{d}\int_{-r_m}^{0} {\mathcal{X}^{\alpha,\alpha}_{\mu}}\left( t,x+r_m \right) \left[F_m\varphi\left( x \right)+A_m \left({_{0^+}^{{C}}\beth^{\alpha}_{\mu}}\varphi\right)\left( x \right) \right] d\mu(x)
\end{align*}
For arbitrary $w,u \in C\left( \left[-r,T  \right], \mathbb{R}^n \right) $, we consider by using Lemma \ref{eql}
\begin{align*}
 \left\Vert  \mathcal{G}w\left(  t\right) - \mathcal{G}u\left(  t\right) \right\Vert  & \leq  \int_{0}^{t} \left\Vert \mathcal{X}^{\alpha,\alpha}_\mu\left(  t,s\right)  \right\Vert \left\Vert  \daleth\left(  s ,  w\left(s  \right)\right)- \daleth\left(  s ,  u\left(s  \right)\right) \right\Vert ds \\
  &= L_{\daleth}\left(\mu(T) - \mu(0)\right) \left \Vert\mathcal{X}^{\alpha,\alpha}_\mu \left(  T,0\right) \right\Vert \left\Vert  w-u \right\Vert_{C}.
\end{align*}
The statements of this theorem ensure that $\mathcal{G}$ is a contraction. By the Banach Contraction principle,  $\mathcal{G}$ is of a unique fixed point on  $\left[ -r, T \right]$, that is $\exists ! w_0 \in  C\left( \left[-r,T  \right], \mathbb{R}^n \right) $, $w_0\left(  t\right)=\mathcal{G}w_0\left(  t\right)$.
\end{proof}

\subsection{Stability of neutral fractional order multi-delayed differential system in the sense of Ulam-Hyers}
We investigate the stability of system (\ref{sorkok2}).

\begin{definition}
Let $\varepsilon>0$. The system (\ref{sorkok2}) is said to be Ulam-Hyers stable if for every solution $w \in C\left( \left[0,T  \right], \mathbb{R}^n \right)$ of inequality,
\begin{equation}\label{sd}
  \left\Vert  {_{0^+}^{{C}}\beth^{\alpha}_{\mu}}   \left[ w \left( t\right)- \sum_{i=1}^{d} A_{i} w \left( t-r_{i}\right) \right]  -Bw\left( t\right)  -\sum_{i=1}^{d}F_{i}w\left(t-r_{i}\right)  -\daleth\left(  t,w(t)\right)   \right\Vert \leq \varepsilon,
\end{equation}
there exists a solution $w_0 \in C\left( \left[0,T  \right], \mathbb{R}^n \right)$ of system (\ref{sorkok2}), and $\eta>0$ such that
\begin{equation*}
   \left\Vert    w \left( t\right)- w_0 \left( t\right) \right\Vert \leq \eta . \varepsilon, \ \ t \in \left[ 0,T \right].
\end{equation*}
\end{definition}

\begin{remark}\label{remarkkk}
A function $w \in C^1\left( \left[0,T  \right], \mathbb{R}^n \right)$ is a solution of the inequality equation (\ref{sd}) if and only if there exists a function $z \in C^1\left( \left[0,T  \right], \mathbb{R}^n \right)$, such that $\left\Vert   z \left( t\right) \right\Vert < \varepsilon$ and ${_{0^+}^{{C}}\beth^{\alpha}_{\mu}}   \left[ w \left( t\right)- \sum_{i=1}^{d} A_{i} w \left( t-r_{i}\right) \right]  =Bw\left( t\right)  +\sum_{i=1}^{d}F_{i}w\left(t-r_{i}\right)  +\daleth\left(  t,w(t)\right) +z \left( t\right)$.
\end{remark}

\begin{theorem}\label{s}
  Suppose that all of the assumptions of Theorem \ref{eu1} are hold. Then system (\ref{sorkok2}) is Ulam-Hyers stable.
\end{theorem}
\begin{proof}
Let   $w \in C\left( \left[0,T  \right], \mathbb{R}^n \right)$ be a solution of the inequality (\ref{sd}), i.e.
\begin{equation}\label{sonthm}
   \left\Vert   {_{0^+}^{{C}}\beth^{\alpha}_{\mu}}   \left[ w \left( t\right)- \sum_{i=1}^{d} A_{i} w \left( t-r_{i}\right) \right]  -Bw\left( t\right)  -\sum_{i=1}^{d}F_{i}w\left(t-r_{i}\right)  --\daleth\left(  t,w(t)\right)  \right\Vert \leq \varepsilon.
\end{equation}
Let  $w_0 \in C\left( \left[0,T  \right], \mathbb{R}^n \right)$ be  the unique solution of system (\ref{sorkok2}), so that
\begin{equation*}
 {_{0^+}^{{C}}\beth^{\alpha}_{\mu}}   \left[ w_0 \left( t\right)- \sum_{i=1}^{d} A_{i} w_0 \left( t-r_{i}\right) \right]  =Bw_0\left( t\right)  +\sum_{i=1}^{d}F_{i}w_0\left(t-r_{i}\right)  +\daleth\left(  t,w_0(t)\right)
\end{equation*}
for each $t\in \left[0,T  \right] $ and $0<\alpha<1$, $w_0\left(  t\right)=w\left(  t\right)$ with $-r \leq t \leq 0$. By combining Remark \ref{remarkkk} and equation (\ref{sonthm}), there exists a function $z \in C\left( \left[0,T  \right], \mathbb{R}^n \right)$  such that  $\left\Vert   z \left( t\right) \right\Vert < \varepsilon$,
 \begin{equation}\label{eqqqqqq}
 {_{0^+}^{{C}}\beth^{\alpha}_{\mu}}   \left[ w \left( t\right)- \sum_{i=1}^{d} A_{i} w \left( t-r_{i}\right) \right]  =Bw\left( t\right)  +\sum_{i=1}^{d}F_{i}w\left(t-r_{i}\right)  +\daleth\left(  t,w(t)\right)  +z \left( t\right).
 \end{equation}
 So we deduced the solution $w\left( t \right)$  from (\ref{eqqqqqq}) with the aid of Corollary \ref{soncor} and the function $\mathcal{G}$ ,
\begin{align*}
    w\left(t  \right)    &=\mathcal{G}w\left(  t\right)+   \int_{0}^{t} {\mathcal{X}^{\alpha,\alpha}_{\mu}}\left(t,x  \right)z\left( x\right)d\mu(x)
\end{align*}
So we have the following estimation
\begin{equation*}
 \left\Vert  \mathcal{G}w\left(  t\right) - w\left(  t\right)\right\Vert \leq  \int_{0}^{t}\left \Vert\mathcal{X}^{\alpha,\alpha}_\mu \left(  t,s\right) \right\Vert \left\Vert z\left(  s \right)\right\Vert d\mu(s) \leq  \left(\mu(T) - \mu(0)\right) \left \Vert\mathcal{X}^{\alpha,\alpha}_\mu \left(  T,0\right) \right\Vert\varepsilon.
\end{equation*}
By the fixed point property of the operator $\mathcal{G}$ given in the proof of Theorem \ref{eu1}, we obtain
\begin{align*}
 \left\Vert w_0\left(  t\right)-  w\left(  t\right)\right\Vert  & \leq  \left\Vert\mathcal{G} w_0\left(  t\right)-  \mathcal{G}w\left(  t\right)\right\Vert +  \left\Vert \mathcal{G}w\left(  t\right)-  w\left(  t\right)\right\Vert \\
   & \leq   L_{\daleth}  \left(\mu(T) - \mu(0)\right) \left \Vert\mathcal{X}^{\alpha,\alpha}_\mu \left(  T,0\right) \right\Vert \left\Vert  w-u \right\Vert_{C} \\
   & + \left(\mu(T) - \mu(0)\right) \left \Vert\mathcal{X}^{\alpha,\alpha}_\mu \left(  T,0\right) \right\Vert\varepsilon.
\end{align*}
By taking $C$-norm on the left hand side and then rearranging the above equation
\begin{equation*}
  \left(1- L_{\daleth} \left(\mu(T) - \mu(0)\right) \left \Vert\mathcal{X}^{\alpha,\alpha}_\mu\left(  T,0\right) \right\Vert  \right)  \left\Vert  w-u \right\Vert_{C} \leq  \left(\mu(T) - \mu(0)\right)\left \Vert\mathcal{X}^{\alpha,\alpha}_\mu \left(  T,0\right) \right\Vert\varepsilon,
\end{equation*}
then we get the desired result
\begin{equation*}
  \left\Vert  w-u\right\Vert_{C} \leq \eta . \varepsilon, \ \ \ \eta = \frac{\left(\mu(T) - \mu(0)\right)\left \Vert\mathcal{X}^{\alpha,\alpha}_\mu \left(  T,0\right) \right\Vert}{  \left(1- L_{\daleth} \left(\mu(T) - \mu(0)\right)\left \Vert\mathcal{X}^{\alpha,\alpha}_\mu \left(  T,0\right) \right\Vert \right) }>0.
\end{equation*}
\end{proof}
\begin{remark}
It is clear as day that all findings are valid when the coefficient matrices are commutative.
\end{remark}

\section{An illustrated example}
We will consider $\sqrt{t}$-Caputo type  fractional neutral differential multi-delayed equations
\begin{equation}
\left\{
\begin{array}
[c]{l}%
{_{0^+}^{{C}}\beth^{0.8}_{\mu}}   \left[ w \left( t\right)-  A_{1} w \left( t-0.3\right)-  A_{2} w \left( t-0.2\right) \right]  =Bw\left( t\right) \\  +F_{1}w\left(t-0.3\right)  +\frac{e^t}{4(1+e^t)}\sin\left( w(t) \right)   ,\ \ t\in\left(  0,0.6\right]  ,\\
\ \ \ \ w\left(  t\right)  =\left(  t^3 \ 2t+1\right)^T  ,\ \ -0.3\leq t\leq0,\\
\end{array}
\right.  \label{examplesorkok}%
\end{equation}
where $\mu(t)=\sqrt{t}$, and
$A_1=\begin{pmatrix}
   0.170 &  0.830 \\
   0&   0.350
   \end{pmatrix}$,
   $ A_2= \begin{pmatrix}
  0.36 & 0.64 \\
  0.07 &   0.11
   \end{pmatrix}$,
    $B=\begin{pmatrix}
   0.33 &  0 \\
   0.03 &   0.125
   \end{pmatrix}$,
    $F_1=\begin{pmatrix}
   0.43 &  470 \\
   0.03 &   0.125
   \end{pmatrix}$,
which are pairwise noncommutative matrices, e.g., $A_1A_2\neq A_2A_1$, $A_1B\neq BA_1$, $A_2B\neq BA_2$, $F_1A_2\neq A_2F_1$,etc. With the well-known maximum absolute row sum of the matrix $\left\Vert . \right\Vert_\infty$; $\left\Vert A_1 \right\Vert_\infty=1$, $\left\Vert A_2 \right\Vert_\infty=1$, $\left\Vert B \right\Vert_\infty=0.33$, and $\left\Vert F_1  \right\Vert_\infty=1$. One can easily check that $\daleth\left(  t,  w(t)\right)=\frac{e^t}{4(1+e^t)}\sin\left( w(t) \right)$ is  continuous in addition being the Lipschitz function with the Lipschitz constant $L_{\daleth}=0.25$   and  $$  L_{\daleth} \left(\mu(0.6) - \mu(0)\right) \mathcal{X}_{\mu,0.8,0.8}^{\left \Vert A_1+A_2 \right\Vert_\infty,\left \Vert B \right\Vert_\infty, \left \Vert F_1\right\Vert_\infty} \left(  0.6,0\right)\cong0.0509175 <1.$$Hence, all of conditions of Theorem \ref{eu1}  and  \ref{s} holds, so system (\ref{examplesorkok}) which has an unique solution is Ulam-Hyers stable.

\section{Conclusion}
In the current paper, we firstly introduce the $\mu$-neutral Caputo type fractional linear multi-delayed differential equations with non-permutable constant coefficient matrices. To obtain a representation of a solution for it, we newly define the $\mu$-neutral multi-delayed perturbation of two parameter Mittag-Leffler type matrix function. We consider the existence and uniqueness of solutions and Ulam-Hyers stability of the $\mu$-neutral fractional multi-delayed differential equations' system.

In terms of qualitative properties and fractional differential equations, this paper includes many different kinds of comprehensive studies\cite{i3}\cite{i5} \cite{6} because for some special cases of $\mu$, we obtain \textbf{the classical Caputo fractional derivative} \cite{sirvastava} \cite{sss1}, Hadamard fractional derivative \cite{sirvastava}, the Caputo–Hadamard fractional derivative \cite{sss2} and the Caputo–Erdélyi–Kober fractional derivative\cite{sss4}.

The next possible further work can be devoted to study asymptotic stability, exponential stability, finite time stability, and also Lyapunov type stability of  the $\mu$-neutral Caputo fractional multi-delayed differential equations with noncommutative coefficient matrices as well as its relative controllability and iterative learning controllability. Another possible direction for additional studies is  to extend our system (\ref{sorkok}) to $\mu$-fractional functional evolution equations and all possibilities as noted just above can be questioned once again for this new system.

\end{document}